\documentclass[12pt]{amsart}

\usepackage{amssymb,latexsym}

\usepackage{enumerate}

\makeatletter

\@namedef{subjclassname@2010}{

  \textup{2010} Mathematics Subject Classification}

\makeatother
\newtheorem{thm}{Theorem}[section]
\newtheorem{Thm}{Theorem}
\newtheorem*{thma}{Theorem A}
\newtheorem*{thmb}{Theorem B}

\newtheorem{lem}[thm]{Lemma}

\theoremstyle{definition}

\numberwithin{equation}{section}

\newcommand{\D}{\mathbb{D}}

\newsymbol\dnd 232D

\renewcommand{\geq}{\geqslant}
\renewcommand{\leq}{\leqslant}
\renewcommand{\mod}[1]{{\ifmmode\text{\rm\ (mod~$#1$)}\else\discretionary{}{}{\hbox{ }}\rm(mod~$#1$)\fi}}
\renewcommand{\Re}{\textup{Re }} 

\begin{document}

\baselineskip=17pt

\title{Lower bounds on odd order character sums}

\author{Leo Goldmakher}
\address{Department of Mathematics, University of Toronto, 40 St. George Street, Toronto, M5S 2E4, Canada}
\email{lgoldmak@math.toronto.edu}

\author{Youness Lamzouri}
\address{Department of Mathematics, University of Illinois at Urbana-Champaign,
1409 W. Green Street,
Urbana, IL, 61801
USA}
\email{lamzouri@math.uiuc.edu}

\date{}

\begin{abstract}
A classical result of Paley shows that there are infinitely many quadratic characters $\chi\mod{q}$ whose character sums get as large as $\sqrt{q}\log \log q$; this implies that a conditional upper bound of Montgomery and Vaughan cannot be improved. In this paper, we derive analogous lower bounds on character sums for characters of odd order, which are best possible in view of the corresponding conditional upper bounds recently obtained by the first author.
\end{abstract}
\subjclass[2010]{Primary 11L40}

\thanks{This work begun while both authors participated at the AMS Mathematics Research Community \emph{``The Pretentious View of Analytic Number Theory''} held on June 26- July 2, 2011 at the Snowbird Resort, Utah. The second author is supported by a postdoctoral fellowship from the Natural Sciences and Engineering Research Council of Canada.}

\maketitle

\section{Introduction}

In this paper we obtain lower bounds on the quantity
\[
M(\chi) := \max_{t} \left|\sum_{n \leq t} \chi(n)\right|
\]
for Dirichlet characters $\chi\mod{q}$ of odd order. The study of character sums and associated quantities such as $M(\chi)$ has been a central topic of analytic number theory for a long time. The first result in this area, discovered independently by P\'{o}lya and Vinogradov in 1918, asserts that
\[
M(\chi) \ll \sqrt{q} \log q
\]
holds uniformly over all characters $\chi\mod{q}$. This bound has resisted any improvement outside of special cases. However, conditionally on the Generalized Riemann Hypothesis, Montgomery and Vaughan \cite{MVExpSumsWMultCoeff} were able to improve this to
\[
M(\chi) \ll \sqrt{q} \log \log q .
\]
For quadratic characters this is known to be optimal, thanks to an unconditional lower bound due to Paley \cite{Paley}. Furthermore, assuming the GRH, Granville and Soundararajan \cite{GS} have extended Paley's lower bound to characters of all even orders.

The story took an unexpected turn when Granville and Soundararajan discovered that both the P\'{o}lya-Vinogradov and the Montgomery-Vaughan bounds can be improved for characters of odd order. In \cite{GS}, they showed that for all characters $\chi\mod{q}$ of odd order $g \geq 3$, there exists $\delta_g > 0$ such that
\[
M(\chi) \ll_g \sqrt{q} (\log q)^{1 - \delta_g + o(1)}
\]
unconditionally and
\[
M(\chi) \ll_g \sqrt{q} (\log \log q)^{1 - \delta_g + o(1)}
\]
conditionally on GRH. (Here $o(1) \to 0$ as $q \to \infty$.) After developing their ideas further, Goldmakher \cite{Go} showed that these bounds hold with
\[
\delta_g = 1 - \frac{g}{\pi} \sin \frac{\pi}{g} .
\]
Moreover, conditionally on GRH, he proved that this value is optimal. To be precise, on GRH he showed that for any $\epsilon >0$ and any fixed odd integer $g \geq 3$, there exist arbitrarily large $q$ and primitive characters $\chi\mod{q}$ of order $g$ satisfying
\begin{equation}
\label{eq:LowerBound}
M(\chi) \gg_{g,\epsilon} \sqrt{q} (\log \log q)^{1 - \delta_g - \epsilon} .
\end{equation}
The goal of the present article is to establish the same result unconditionally.

Recent progress on character sums was made possible by Granville and Soundararajan's discovery that $M(\chi)$ depends on the extent to which $\chi$ mimics the behavior of other characters. To measure this mimicry, they introduced the symbol
\[
\D(\chi, \psi; y) :=
    \left( \sum_{p \leq y} \frac{1 - \Re \chi(p) \overline{\psi(p)}}{p} \right)^{1/2} .
\]
It turns out that this defines a pseudometric on the space of characters, and has a number of interesting properties; see \cite{GSCuriousZetaInequalities} for an in-depth discussion. In \cite{GS}, Granville and Soundararajan derive a number of upper and lower bounds on $M(\chi)$ in terms of $\D(\chi, \psi; y)$, where $\psi$ is a character of small conductor and opposite parity to $\chi$ (i.e.\ $\chi(-1)\psi(-1) = -1$). For example, they prove the following:
\begin{thma}[Theorem 2.5 of \cite{GS}]
Assume GRH. Let $\chi\mod{q}$ and $\psi\mod{m}$ be primitive characters with $\psi(-1)=-\chi(-1)$. Then
\[
M(\chi)+\frac{\sqrt{qm}}{\phi(m)}\log\log\log q
 \gg
\frac{\sqrt{qm}}{\phi(m)}\log\log q\exp\left(-\D(\chi, \psi;\log q)^2\right).
\]
\end{thma}

In this way, the problem of bounding $M(\chi)$ is translated into that of bounding $\D(\chi,\psi;\log q)$. A consequence of Lemma 3.2 of \cite{GS} is that whenever $\chi\mod{q}$ is a primitive character of odd order $g$ and $\xi\mod{m}$ is a primitive character of opposite parity and small conductor ($m\leq (\log\log q)^A$),
\[
\D(\chi, \xi; \log q)^2 \geq \big(\delta_g + o(1)\big) \log \log \log q .
\]
Goldmakher, using a reciprocity law for the $g^\text{th}$-order residue symbol and the Chinese Remainder Theorem, proved a complementary result:
\begin{thmb}[\cite{Go}, Section 9]
Let $g\geq 3$ be a fixed odd integer. For any $\epsilon > 0$ there exists an odd character $\xi\mod{m}$ with $m\ll_{\epsilon}1$ and an infinite
family of primitive characters $\chi\mod{q}$ of order $g$ such that
\[
\D(\chi, \xi; \log q)^2 \leq (\delta_g + \epsilon) \log \log \log q .
\]
\end{thmb}
\noindent
Combining Theorems A and B produces the lower bound (\ref{eq:LowerBound}) on $M(\chi)$; conditionally, since Theorem A is dependent on the GRH.

To prove (\ref{eq:LowerBound}) unconditionally, we derive an unconditional version of Theorem A. Although we cannot prove a totally analogous theorem, in the case where $\chi$ is even we are able to remove the assumption of GRH at a cost of an extra $(\log\log\log q)^{-1}$ factor:
\begin{Thm}
\label{thm:MainThm}
Let $\chi\mod{q}$ be a primitive even character and $\psi\mod{m}$ be a primitive odd character. Then
\[
M(\chi)+\sqrt{q}
\gg
\frac{\sqrt{qm}\log\log q}{\phi(m)\log\log\log q}\exp\left(-\D(\chi, \psi;\log q)^2\right).
\]
\end{Thm}
\noindent
Combining Theorem \ref{thm:MainThm} with Theorem B, we deduce the desired lower bound.
\begin{Thm}
Let $g\geq 3$ be a fixed odd integer. There exist arbitrarily large $q$ and primitive characters $\chi\mod{q}$ of order $g$ such that
\[
M(\chi)
\gg_{g,\epsilon}
\sqrt{q}(\log\log q)^{1-\delta_g-\epsilon},
\]
where $\delta_g=1-\frac{g}{\pi}\sin \frac{\pi}{g}.$
\end{Thm}

\section{The key lemmas}

In this section, we prove two general results which will be the main ingredients in the proof of Theorem 1, and might also be of independent interest.
\begin{lem}
\label{lem:2-1}
Let $f$ be a multiplicative function such that $|f(n)|\leq 1$ for all $n\geq 1$, and let $y$ be a large real number. Then
\[
\max_{N\leq y}\left|\sum_{n\leq N}\frac{f(n)}{n}\right|+\frac{1}{\log\log y}
\gg
\frac{\log y}{\log\log y}\exp\left(-\D(f,1;y)^2\right).
\]
\end{lem}
\emph{Remark.} This lemma is a generalization of Lemma 6.3 of \cite{GS}, where $y$ is a positive integer and $f$ is a primitive character modulo $y$. Note that in this special case the factor $(\log\log y)^{-1}$ can be removed from the RHS of the above inequality.
\begin{proof} Let $\delta>0$ be a real number to be chosen later. Then
\[
 \int_1^y\frac{\delta}{t^{1+\delta}}\sum_{n\leq t}\frac{f(n)}{n}dt
 =\sum_{n\leq y}\frac{f(n)}{n}\int_n^y\frac{\delta}{t^{1+\delta}}dt=
 \sum_{n\leq y}\frac{f(n)}{n^{1+\delta}}-y^{-\delta}\sum_{n\leq y}\frac{f(n)}{n}.
\]
Hence, using that $\int_1^y\frac{\delta}{t^{1+\delta}}dt=1-y^{-\delta}$, we derive
\begin{equation}
\label{eq:2-1}
\left|\sum_{n\leq y}\frac{f(n)}{n^{1+\delta}}\right|
\leq
\max_{N\leq y}\left|\sum_{n\leq N}\frac{f(n)}{n}\right|.
\end{equation}
For $\text{Re}(s)>1$, let $F(s)=\sum_{n=1}^{\infty}f(n)/n^{s}$ be the Dirichlet series of $f$. Then
\[
F(1+\delta)
=\sum_{n\leq y}\frac{f(n)}{n^{1+\delta}}+ O\left(\sum_{n>y}\frac{1}{n^{1+\delta}}\right)
= \sum_{n\leq y}\frac{f(n)}{n^{1+\delta}}+ O\left(\frac{1}{\delta y^{\delta}}\right).
\]
We choose $\delta=\log\log y/\log y$, which yields
\begin{equation}
\label{eq:2-2}
F(1+\delta)
=\sum_{n\leq y}\frac{f(n)}{n^{1+\delta}}+O\left(\frac{1}{\log\log y}\right).
\end{equation}
On the other hand, we obtain from the Euler product of $F(s)$
\[
\log F(1+\delta)
=\sum_{p}\frac{f(p)}{p^{1+\delta}}+O(1)
= \sum_{p\leq e^{1/\delta}}\frac{f(p)}{p^{1+\delta}}+O(1)
= \sum_{p\leq e^{1/\delta}}\frac{f(p)}{p}+O(1).
\]
Therefore, we get
\[
\left|\log|F(1+\delta)|-\sum_{p\leq y}\frac{\text{Re}f(p)}{p}\right|
\leq
\sum_{e^{1/\delta}< p\leq y}\frac{1}{p} + O(1)
=
\log\log\log y + O(1)
\]
which implies
\[
 |F(1+\delta)|
 \gg
 \frac{\log y}{\log\log y}\exp\left(-\D(f,1;y)^2\right).
\]
The lemma follows upon combining this last estimate with (\ref{eq:2-1}) and (\ref{eq:2-2}).
\end{proof}

Our second lemma is inspired by Paley's approach in \cite{Paley}.  Recall that the \emph{Fej\'er kernel}, defined by
\[
F_N(\theta)
:=\sum_{|n|\leq N}\left(1-\frac{|n|}{N}\right)e(n\theta),
\]
 where $e(t)=e^{2\pi i t}$, satisfies
\[
F_N(\theta)
= \frac{1}{N} \left(\frac{\sin(\pi N\theta)}{\sin(\pi\theta)}\right)^2
\geq 0
\quad
\text{ and }
\quad
\int_0^1 F_N(\theta)d\theta=1.
\]
Using these properties of the Fej\'er kernel, we establish the following lemma:
\begin{lem}
\label{lem:2-2}
Let $\{a(n)\}_{n\in \mathbb{Z}}$ be a sequence of complex numbers with $|a(n)|\leq 1$ for all $n$, and let $x\geq 2$ be a real number.
Then
\[
\max_{\theta\in [0,1]}\max_{1\leq N\leq x}\left|\sum_{1\leq |n|\leq N}\frac{a(n)}{n} e(n\theta)\right|
=\max_{\theta\in [0,1]}\left|\sum_{1\leq |n|\leq x}\frac{a(n)}{n} e(n\theta)\right|+O(1).
\]
\end{lem}
\begin{proof}
To establish the result, we only need to prove the implicit upper bound, since the implicit lower bound holds trivially. Let $\alpha\in [0,1]$, and $1\leq N\leq x$. First, note that
\[
\sum_{1\leq |n|\leq N}\frac{a(n)}{n} e(n\alpha)
=\sum_{1\leq |n|\leq N}\frac{a(n)}{n} e(n\alpha)\left(1-\frac{|n|}{N}\right)+O(1).
\]
Moreover, we have
\begin{equation*}
\begin{aligned}
\sum_{1\leq |n|\leq N}\frac{a(n)}{n} e(n\alpha)\left(1-\frac{|n|}{N}\right)
&= \sum_{1\leq |n|\leq x}\frac{a(n)}{n} e(n\alpha)\sum_{|m|\leq N}\left(1-\frac{|m|}{N}\right)\int_0^1e(m\theta)e(-n\theta)d\theta\\
&=\int_0^1\left(\sum_{1\leq |n|\leq x}\frac{a(n)}{n} e\big(n(\alpha-\theta)\big)\right)F_N(\theta)d\theta.
\end{aligned}
\end{equation*}
Thus, we obtain
\[
\left|\sum_{1\leq |n| \leq N}\frac{a(n)}{n} e(n\alpha)\right|
\leq
\max_{\theta\in [0,1]}\left|\sum_{1\leq |n|\leq x}\frac{a(n)}{n} e(n\theta)\right|+O(1).
\]
Since $\alpha$ and $N$ were arbitrary the result follows.
\end{proof}

\section{Proof of Theorem 1}

Given a primitive character $\chi\mod{q}$, recall P\'olya's fourier expansion \cite{MVbook}:
\[
\sum_{n\leq t}\chi(n)
= \frac{\tau(\chi)}{2\pi i}\sum_{1\leq |n|\leq q}\frac{\overline{\chi}(n)}{n}\left(1-e\left(-\frac{nt}{q}\right)\right)+O(\log q),
\]
where
\[
\tau(\chi):=\sum_{b\mod{q}} \chi(b)e(b/q),
\]
is the \emph{Gauss sum}.
It follows that if $\chi$ is even,
\[
\sum_{n\leq t}\chi(n)
= -\frac{\tau(\chi)}{2\pi i}\sum_{1\leq |n|\leq q}\frac{\overline{\chi}(n)}{n}e\left(-\frac{nt}{q}\right)+O(\log q),
\]
and therefore we get in this case
\begin{equation}
\label{eq:3-1}
M(\chi)+\log q
\gg
\sqrt{q}\max_{\theta\in [0,1]}\left|\sum_{1\leq |n|\leq q}\frac{\chi(n)}{n}e(n\theta)\right|.
\end{equation}

\begin{proof}[Proof of Theorem \ref{thm:MainThm}] First, we infer from Lemma \ref{lem:2-2} that
\begin{equation}
\label{eq:3-2}
\max_{\theta\in [0,1]}\left|\sum_{1\leq |n|\leq q}\frac{\chi(n)}{n}e(n\theta)\right|
= \max_{\theta\in [0,1]}\max_{N\leq q}\left|\sum_{1\leq |n|\leq N}\frac{\chi(n)}{n}e(n\theta)\right|+O(1).
\end{equation}
Now, using that
\[
\sum_{ b\mod{m}}\psi(b)e(bn/m)=\overline{\psi}(n)\tau(\psi),
\]
for all integers $n$ (since $\psi$ is primitive, see for example \cite{Da}) we get
\begin{align*}
\sum_{b\mod{m}} \psi(b)\left(\sum_{1\leq |n|\leq N}\frac{\chi(n)}{n}e\left(\frac{n b}{m}\right)\right)&=
\sum_{b\mod{m}} \psi(b)\left(\sum_{n\leq N}\frac{\chi(n)}{n}\left(e\left(\frac{n b}{m}\right)-e\left(\frac{-n b}{m}\right)\right)\right)\\
&= 2\tau(\psi)\sum_{n\leq N}\frac{\chi(n)\overline{\psi}(n)}{n},
\end{align*}
since $\psi(-1)=-1$. We therefore obtain
\[
\max_{\theta\in [0,1]}\left|\sum_{1\leq |n|\leq N}\frac{\chi(n)}{n}e(n\theta)\right|\gg \frac{\sqrt{m}}{\phi(m)}\left|\sum_{n\leq N}\frac{\chi(n)\overline{\psi}(n)}{n}\right|.
\]
Combining this estimate with (\ref{eq:3-1}) and (\ref{eq:3-2}) we deduce
\[
M(\chi)+\sqrt{q}\gg \frac{\sqrt{qm}}{\phi(m)}\max_{N\leq q}\left|\sum_{n\leq N}\frac{\chi(n)\overline{\psi}(n)}{n}\right|.
\]
Finally, appealing to Lemma \ref{lem:2-1} we find
\[
\max_{N \leq q}
    \left|\sum_{n\leq N}\frac{\chi(n)\overline{\psi}(n)}{n}\right| + 1
\geq
\max_{N \leq \log q}
    \left|\sum_{n\leq N}\frac{\chi(n)\overline{\psi}(n)}{n}\right| + 1
\gg
    \frac{\log\log q}{\log\log\log q} \exp\left(-\D(\chi,\psi;\log q)^2\right),
\]
which completes the proof.
\end{proof}

\end{document}